\documentclass[12pt]{article}
\usepackage[english]{babel}
\usepackage{amsmath,amssymb,hyperref,xcolor,latexsym,theorem}
\usepackage{setspace}
\usepackage[a4paper]{geometry}
\usepackage{dsfont}
\newcommand{\nobracket}{}

\newcommand{\tmmathbf}[1]{\ensuremath{\boldsymbol{#1}}}
\newcommand{\tmop}[1]{\ensuremath{\operatorname{#1}}}
\newcommand{\tmtextbf}[1]{\text{{\bfseries{#1}}}}
\newcommand{\tmtextit}[1]{\text{{\itshape{#1}}}}
\newenvironment{proof}{\noindent\textbf{Proof\ }}{\hspace*{\fill}$\Box$\medskip}
{\theorembodyfont{\rmfamily}\newtheorem{example}{Example}}
\newtheorem{theorem}{Theorem}
\newtheorem{claim}{Claim}
\newtheorem{definition}{Definition}
\newtheorem{remark}{Remark}
\newtheorem{assumption}{Assumption}
\newcommand{\dd}{\mathrm d}

\author{
Maximilian Engel\thanks{
Freie Universit\"at Berlin, Institut f\"ur Mathematik und Informatik, 14195 Berlin, Germany and
University of Amsterdam, KdV Institute for Mathetmatics, 1098 XG
Amsterdam, Netherlands, 
email: m.r.engel@uva.nl.
} \quad 
Peter K.\ Friz\thanks{
Technische Universit\"at Berlin and 
Weierstrass Institut (WIAS), Berlin, Germany, 
email: friz@math.tu-berlin.de.
} \quad 
Tal Orenshtein\thanks{
Università degli Studi di Milano-Bicocca, Dipartimento di Matematica e Applicazioni, Via Cozzi 55, 20125 Milano, Italy, 
email: tal.orenshtein@unimib.it.
}
}

\begin{document}

\onehalfspacing

\title{Nonlinear effects within invariance principles}

\maketitle

\begin{abstract}
The combination of functional limit theorems with the pathwise analysis of deterministic and stochastic differential equations has proven to be a powerful approach to the analysis of fast-slow systems. In a multivariate setting, this requires rough path ideas, as already suggested in the seminal work [Melbourne-Stuart, Nonlinearity, 24, 2011]. 
This initiated a program pursued by numerous authors and which has required substantial results on invariance principles (also known as functional central limit theorems) in rough path topologies. We take a unified point of view and provide simple and exact formulas, of the Green-Kubo type, that characterize the relevant Brownian rough path limits and discuss how they naturally apply in different settings.
\end{abstract}

\section{Introduction}

Random walks are stochastic processes in which a particle undergoes successive, independent, and identically distributed steps in various directions. When these steps are small and frequent, the cumulative effect over time resembles the continuous and erratic motion observed in Brownian motion in rather general settings. The (weak) convergence of random walks whose jumps have finite variance and vanishing mean, to Brownian motion, also known as the functional central limit theorem, is formalized by Donsker's invariance principle. This principle is well understood, and in particular, the conclusion remains valid even when the assumption of independent steps is relaxed to suitable mixing conditions that quantify correlation. 
Classical references include 
Hall-Heyde 1980 \cite{hall1980martingale} and 
Theorem 1 in Doukhan et al 1994 \cite{doukhan1994functional}.

%%\subsection{Discussion}
%PF Aug-27: 
%
%Hall \& Hedyde 1980, \\
%\url{http://www.stat.yale.edu/~mjk56/MartingaleLimitTheoryAndItsApplication.pdf}
%
%Thm 1 in Doukhan et al 1994 
%\url{http://www.numdam.org/item/AIHPB_1994__30_1_63_0.pdf}
%
%Might have some useful refs
%\url{http://galton.uchicago.edu/~wbwu/papers/SII139.pdf}

From Donsker's perspective, no matter if the random walk (rescaled) is a piecewise constant process or a piecewise linear process - the resulting (weak) limit, Brownian motion, remains identical. However, if employed as a driving noise of a dynamical system this is no longer the case, and hence, from a modelling point of view, it requires (or permits) a choice: 
%More subtle, if classical, is the effect in combination with stochastic differential equations (SDEs). 
%From Donsker’s perspective, there is little difference in regarding the random walk as a piecewise constant process or a piecewise linear process since the resulting (rescaled, weak) limit, Brownian motion, is identical. 
The first one, common in financial mathematics, leads to stochastic recursions that converge weakly to It\^{o} stochastic differential equations (ISDEs).\footnote{Within stochastic analysis, the so-called UT/UCV theory will cover this situation nicely.} The second one, commonly used by physicists and geometers, leads to random ODEs with weak ``Wong-Zakai'' limits described by Stratonovich stochastic differential equations (SSDEs).

\subsection*{Related work on limit SDEs}

Classical stochastic analysis treatises such as Ikeda-Watanabe \cite{ikeda1981stochastic} consider SSDEs as drift-perturbed ISDEs; a geometric ``Wong-Zakai'' approach is advocated by Stroock
\cite{stroock2003markov}. The latter has been given a profound understanding through Lyons' theory of rough paths which reduces the Wong-Zakai results to establishing 
convergence of piecewise linear approximations at the level of process and L\'evy area; an application to support theory was discussed in \cite{friz2006levy}.

Interestingly, other approximations, like those induced by considering charged Brownian particles in a magnetic field \cite{friz2015physical} lead to limit SDEs which are in general of neither It\^{o} nor Stratonovich type.
In the (somewhat restricted) case of exact coefficient fields (think: additive noise or diffeomorphic transformations thereof), the same structure appears for an intriguing class of fast processes - that play the role of our rescaled random walks. This was explicitly pointed out by Melbourne and Stuart \cite{melbourne2011note}. 

In a celebrated series of works, started with \cite{kelly2016smooth} and  \cite{kelly2017deterministic}, exactness was removed, thanks to rough path theory. In later joint work with one of the authors
\cite{chevyrev2022deterministic}, we obtained validity of these results under optimal moment assumptions. We also draw attention to the recent work by Gottwald and Melbourne \cite{Gottwald_2024} where the role of L\'evy area and possible corrections are discussed, fully aligned with the rough path view taken in this work.

We would also like to point to more results, in the Markovian setting, by Deuschel-Orenshtein-Perkowski \cite{deuschel2021additive} and the one in the regenerative / i.i.d-like setting, by Orenshtein \cite{orenshtein2021rough} where the focus is on the relevant rough paths results; for the non-Markovian situation and almost-sure results
see Friz-Kifer \cite{friz2021almost}.

All these cited works yield, in different settings and with explicit formulae, that the limiting dynamics are, in general, of neither It\^{o} nor Stratonovich type.

\subsection*{Main insights and the structure of the paper}

The purpose of this short note, deliberately non-technical, at least in comparison to most of the aforementioned works, is to point out rather easy second-moment computations that, for
all of the aforementioned situations, lead to the correct formula of Green-Kubo type; such formulas are familiar to people with knowledge in homogenization of stochastic processes, as discussed for instance in 
the book by Pavliotis and Stuart \cite{pavliotis2008multiscale}, which describes correctly the structure of the limit Brownian rough paths, and hence settles the limiting
stochastic dynamics of accompanying fast-slow systems.

To avoid misunderstandings, we do NOT offer to newly (re-)prove results in \cite{deuschel2021additive, friz2015physical, friz2021almost, kelly2016smooth, kelly2017deterministic, orenshtein2021rough} 
and in particular shall not give a general proof of weak convergence. Our intention is to point out that a sufficiently nice convergence, in whatever setting, leads to the same structure of the limit Brownian rough paths.

In the following, after introducing the appropriate general set-up in which our results will be formulated, we discuss iterated invariance principles for processes with stationary increments in different situations:
%, where the structure is as follows: 
each section contains a main theorem followed by various relevant examples. Section 2.1 deals with the discrete time case, with results on problems of It\^o type in Section 2.1.1 and those of Wong-Zakai type in Section 2.1.2. 
Continuous time iterated invariance principles are discussed in Section 2.2: we deal with stationary continuous time processes in Section 2.2.1 and discuss continuous time suspension flows, which are built from discrete settings in Section 2.2.2. 

%We remark that in the case of charged Brownian particles in a magnetic field \cite{friz2015physical}, Example 6 below, can be viewed as a special case of the Markovian settings of Deuschel-Orenshtein-Perkowski \cite{deuschel2021additive}, Example 7 below, this is demonstrated in Example 8 below. 
%ME: I think the paragraph above is to detailed for an introduction

Finally, we remark that some of the examples in the literature are non-stationary (to be understood in the appropriate sense). Hence, a reoccurring feature of our paper is the adjustment of the various models to the stationary settings. This requires an additional argument which may be of interest beyond the problem of iterated invariance principles: see, for instance, Theorem \ref{thm:suspension} and Example \ref{example:regeneration}.

\section{Iterated invariance principles in different settings}

A Brownian rough path over $\mathbb{R}^d$ with characteristics
$(\Sigma, \Gamma) \in (\mathbb{R}^d)^{\otimes 2} \times
(\mathbb{R}^d)^{\otimes 2}$
is any process defined by $t \mapsto \tmmathbf{B}(t)
= (B(t), \mathbb{B}(t))$, where $B$ is a $d$-dimensional Brownian motion with
covariance
\[ \Sigma = \tmmathbf{E} (B (1) \otimes B (1)) \in (\mathbb{R}^d)^{\otimes 2}
\]
and $\mathbb{B}$ is of the form
\[ \mathbb{B} (t) = \int_0^t B  \otimes \dd B + t \Gamma \in
   (\mathbb{R}^d)^{\otimes 2} , \]
that is \[ \mathbb{B}^{i,j}(t) = \int_0^t B ^{i}(s) \dd B^{j}(s) + t \Gamma ^{i,j}, \; i,j=1,...,d.\]
The integral above is in the sense of It\^{o} and $\tmmathbf{E} (X)$ is defined component-wise, whenever $X$ is a random vector or a random matrix.
In particular,    
$\Gamma = \tmmathbf{E} (\mathbb{B} (1))$.
We recall some well-known strong limit theorems. The It\^{o} Brownian rough path
($\Gamma = 0$) arises from considering lifted (càdlàg) piecewise constant
approximations which is tantamount to taking  left-point Riemann-Stieltjes approximations of $\int B \otimes \dd B$. 
The Stratonovich Brownian rough path ($\Gamma = \frac{1}{2}
\Sigma$) \ arises e.g.~from piecewise approximations, corresponding with mid-point Riemann-Stieltjes approximations of $\int B \otimes \dd B$. 
All mentioned convergences are \tmtextit{pointwise}, i.e. for fixed $t$, and in $\tmmathbf{P}$-probabiliy, but can be upgraded to convergence in appropriate
rough path metrics and\tmtextbf{} $L^{\infty -} (\tmmathbf{P})$. Below we will
be interested in weak, a.k.a.~in law, limit theorems.

For generality, we shall consider the following definition.
\begin{definition}
\label{def:main}
    Let $t \mapsto (S_N(t), \mathbb{S}_N(t))$ be a random process on $\mathbb{R}^d \times (\mathbb{R}^d)^{\otimes 2}$. We say that the iterated invariance principle holds for $(S_N, \mathbb{S}_N)$ with a Brownian
  rough path limit $(B, \mathbb{B})$ over $\mathbb{R}^d$ 
if the following two conditions hold.
    \begin{enumerate}
        \item For all $t$ the $\mathbb{R}^d \times (\mathbb{R}^d)^{\otimes 2}$-valued sequence $(S_N(t), \mathbb{S}_N(t))$ converges in law to $(B(t),\mathbb{B}(t))$.  
        \item  $\tmmathbf{E}(S_N(1)\otimes S_N(1))$ converges to $\tmmathbf{E}(B(1)\otimes B(1))$ and $\tmmathbf{E}(\mathbb{S}_N(1))$ converges to $\tmmathbf{E}(\mathbb{B}(1))$. 
%\item  $\tmmathbf{E}(S_N(1)\otimes S_N(1))$ converges to $\Sigma$ and $\tmmathbf{E}(\mathbb{S}_N(1))$ converges to $\Gamma$, where $(\Sigma, \Gamma) %\in (\mathbb{R}^d)^{\otimes 2} \times (\mathbb{R}^d)^{\otimes 2}$ are the characteristics of $(B, \mathbb{B})$. 
\end{enumerate}
\end{definition}
We refer to Ekisheva-Houdré \cite{ekisheva2006transportation} (and the references therein) which contains various conditions to the central limit theorem together with convergence of moments.   
 
%For example, Conditions 1 and 2 are satisfied if  $W_2(\mu_N(t),\mu(t))\to 0$ for all $t$ fixed, where $W_2$ is the Wasserstein 2-distance and $\mu_N(t)$ and $\mu(t)$ are the laws of $(S_N(t),\mathbb{S}_N(t))$ and  $(B(t),\mathbb{B}(t))$ on $(\mathbb{R}^d)^{\otimes 2} \times (\mathbb{R}^d)^{\otimes 2}$, see [Ekisheva-Houdré below line (3) and the references therein], or [Bickel and Freedman].
%\tnote{I did not find an elegant/compact way to define a norm of the 2-levels jointly so that we take the 1-Wasserstein distance and use use the paragraph after line (3) of Ekisheva-Houdré to get a sufficient condition to our Definition 1 above. Peter, if you see how to do it, then it would be nice, otherwise feel free to erase my remark.}

\subsection{Discrete iterated invariance principles}
\subsubsection{It\^{o}-type iterated invariance principles}
Consider a $d$-dimensional discrete-time stationary process $\xi = \xi (k ;
\omega)$ with rescaled partial sums
\begin{equation} \label{eq:S_N}
    S_N (t) : = \frac{1}{N^{1 / 2}} \sum_{0 \leqslant k < [N t]} \xi (k).
\end{equation} 
Assume
$\tmmathbf{}
\tmmathbf{E} \xi (0) = 0$ and consider $S_N$ with the canonical enhancement of iterated
sums $\mathbb{S}_N$ defined by 
\begin{equation} \label{eq:mathhbb S_N}
    \mathbb{S}_N (t) : = \frac{1}{N} \sum_{0 \leqslant k < \ell < [N t]} (\xi (k)
   \otimes \xi (\ell)).
\end{equation} 

\begin{theorem}
\label{thm:Main_thm_discrete}
  Let $\xi$ be a discrete stationary process with finite second moments and
  summable correlations, that is
  \[ \Delta (0) \quad \text{ and } \quad \sum_{n =
     1}^{\infty} \Delta (n) \text{ are finite}, \]
     where $\Delta (n) : = \tmmathbf{E} (\xi (0) \otimes \xi (n))$ for $n=0,1,...$.    
  Consider the rescaled piecewise constant approximations $S_N$ \eqref{eq:S_N} with the iterated sum process $\mathbb{S}_N$ \eqref{eq:mathhbb S_N}. If the iterated invariance principle holds for $(S_N,\mathbb{S}_N)$, then the Brownian
  rough path limit $(B, \mathbb{B})$ is necessarily with characteristics
  \[ \Sigma = \Delta (0) + 2 \tmop{Sym} (\Gamma) \quad \text{ and } \quad \Gamma = \sum_{n = 1}^{\infty} \Delta (n),\]
where $\tmop{Sym}(\Gamma)$ denotes the symmetric part of the matrix $\Gamma$.
\end{theorem}
The expression for the covariance matrix $\Sigma$ is sometimes called \emph{Green-Kubo formula}.

\begin{proof}
Let $\sigma_n := \sum_{k=1}^n \Delta (k)$. Using stationarity, we see that

\begin{align*}
    \sum_{0 \leqslant k< \ell < N-1} \tmmathbf{E} (\xi
   (k) \otimes \xi (\ell)) =
    \sum_{0 \leqslant k< \ell < N-1} \Delta (\ell-k)  
    =  \sum_{1\leq k < N} (N-k) \Delta(k) = \sum_{n=1}^N \sigma_n.
\end{align*}
The summability of the correlations $\Delta (n)$ implies that
\begin{align*}
 \lim_{N\to \infty} \frac{1}{N} \sum_{n=1}^N \sigma_n
   &= \lim_{N\to \infty} \sigma_N = \sum_{k=1}^{\infty} \Delta (k).
\end{align*}
Hence 
\begin{align*}
   \lim_{N\to \infty} \frac{1}{N} \sum_{0 \leqslant k< \ell < N-1} \tmmathbf{E} (\xi
   (k) \otimes \xi (\ell))  &=   
   \sum_{k=1}^{\infty} \Delta (k)
\end{align*}
and analogously, 
\begin{align*}
      \lim_{N\to \infty} \frac{1}{N} \sum_{0 \leqslant k< \ell < N-1} \tmmathbf{E} (\xi
   (\ell) \otimes \xi (k))  &= \sum_{k=1}^{\infty} \Delta (k)^T,
\end{align*}
where $\Delta (n)^T = \tmmathbf{E} (\xi (n) \otimes \xi (0))$.
Therefore
\begin{align*}
    \Sigma & \approx  \tmmathbf{E} (S_N(1)\otimes S_N(1))  = \frac{1}{N} \sum_{0 \leqslant k, \ell < N-1} \tmmathbf{E} (\xi
   (k) \otimes \xi (\ell)) \\
   &\approx \tmmathbf{E} (\xi (0) \otimes \xi (0)) +
   \sum_{n = 1}^{\infty} \tmmathbf{E} (\xi (0) \otimes \xi (n) + \xi (n)
   \otimes \xi (0)),
\end{align*}  
where here and after we use the notation $a_N \approx b_N$ if $a_N-b_N \to 0$ as $N$ tends to infinity.  

In terms of $\Delta (n)$, this becomes
$\Sigma = \Delta (0) + 2 \sum_{n = 1}^{\infty} \tmop{Sym} (\Delta (n))$. 
On the other hand, 
$\tmmathbf{E} (\mathbb{S}_N (1))  \approx \tmmathbf{E} (\mathbb{B} (1))=\Gamma$.
The computation above gives
\[ \Gamma \approx \frac{1}{N} \sum_{0 \leqslant k < \ell < N} \tmmathbf{E}
   (\xi (k) \otimes \xi (\ell)) \approx 
   \sum_{n = 1}^{\infty} \Delta (n) . \]
\end{proof}

\begin{example}
  Kelly-Melbourne \cite{kelly2016smooth, kelly2017deterministic} and (8),(19) in Chevyrev et al \cite{chevyrev2019multiscale}:
 % ~\href{https://arxiv.org/pdf/1712.01343.pdf}{https://arxiv.org/pdf/1712.01343.pdf}.
  Considering maps $T : M \rightarrow M$ on a compact manifold $M$ with invariant and ergodic measure $\mu$ on a hyperbolic invariant set $\Lambda \subset M$, Theorem~\ref{thm:Main_thm_discrete} applies with
  \[ \Delta (n) : = \tmmathbf{E} (\xi (0) \otimes \xi (n)) = \int (v \otimes
     (v \circ T^n)) \dd \mu, \]
   for a $\mu$-centered H\"older observable $v : \Lambda \rightarrow
  \mathbb{R}^d$, where the assumption of Theorem~\ref{thm:Main_thm_discrete} corresponds with summable decay of correlations, i.e.~sufficiently fast mixing. 
  The probability space in this case is thus $(\Lambda, \mu)$ and the stationary process is $\xi (n ; \omega) = (v \circ T^n) (\omega)$. Notable examples are uniformly expanding maps $T$ and non-uniformly expanding maps $T$ with appropriate Young towers.
\end{example}

\begin{example}
  Kurtz-Protter \cite{kurtz1991wong}. The $\xi$ are taken IID, centred with finite second
  moments. Then $\Gamma=0$ and one gets the expected It\^{o} limit. (The $S_N$ are
  martingales, and satisfy the UCV condition.) This already implies tightness
  in rough paths metrics, cf.\ \cite{chevyrev2019canonical} and also in cadlag rough path spaces
  \cite{friz2018differential}.
\end{example}

\begin{example}
  Theorem 2.2.\ in Friz-Kifer \cite{friz2021almost}.
%  \href{https://arxiv.org/pdf/1712.01343.pdf}{https://arxiv.org/pdf/2111.05390.pdf}
  Direct mixing conditions are given. With
  \[ \Delta (n) : = \tmmathbf{E} (\xi (0) \otimes \xi (n)) \]
  an almost sure iterated invariance principle is shown, with $\Sigma,
  \Gamma$ as described by Theorem~\ref{thm:continuous_time}.
\end{example}

\subsubsection{Wong-Zakai type iterated invariance principles}

We maintain the discrete setup, but do not work with the piecewise constant
partial sums $S_N$ but instead, \`{a} la Wong-Zakai, with their piecewise-linear
counterpart $\hat{S}_N$. Write $(\hat{S}_N, \hat{\mathbb{S}}_N)$ for the
canoncial (geometric) rough path lift. More explicitly,
\begin{equation} \label{eq:hat_S_N}
     \hat{\mathbb S}_N (t) = \int_0^{t} \hat{S}_N(s) \otimes \dd \hat{S}_N(s),
\end{equation}
where the integration is in the Riemann–Stieltjes sense.

\begin{theorem}
  In the assumptions of Theorem \ref{thm:Main_thm_discrete}, if the iterated invariance principle holds for  $(\hat{S}_N, \hat{\mathbb{S}}_N)$ (defined above, see \eqref{eq:hat_S_N}), then the (Stratonovich)
  Brownian rough path limit $(B, \hat{\mathbb{B}})$ has characteristics
  $(\Sigma, \hat{\Gamma})$,
  \[ \Sigma = \Delta (0) + 2 \tmop{Sym} \left( \sum_{n = 1}^{\infty} \Delta
     (n) \right), \qquad \hat{\Gamma} = \frac{1}{2} \Delta (0) + \sum_{n =
     1}^{\infty} \Delta (n) . \]
  As a consequence,
  \[ \hat{\mathbb{B}} (1) = \int_0^1 B \otimes \circ \dd B + \tmop{Anti} \left( \sum_{n =
     1}^{\infty} \Delta (n) \right), \]
where the last integration is in the sense of Stratonovich.
\end{theorem}

\begin{proof}
  The covariance $\Sigma$ is not affected by the problem modification, hence
  as before. We have a discrete It\^{o}-Stratonovich correction of the form
  \[ \hat{\mathbb{S}}_N (1) =\mathbb{S}_N (1) + \frac{1}{2 N } \sum_{0
     \leqslant k < N} \xi (k) \otimes \xi (k) \quad \]
     and, hence, 
     \[ \tmmathbf{E}
     (\hat{\mathbb{S}}_N (1)) = \tmmathbf{E} (\mathbb{S}_N (1)) + \frac{1}{2}
     \Delta (0). \]
  The argument in the proof of Theorem \ref{thm:Main_thm_discrete} gives
  \[ \hat{\Gamma} = \tmmathbf{E} (\hat{\mathbb{B}} (1)) = \tmmathbf{E}
     (\mathbb{B} (1)) + \frac{1}{2} \Delta (0) = \Gamma + \frac{1}{2} \Delta
     (0) . \]
  Note that
  \[ \hat{\mathbb{B}} (1) = \int_0^1 B \otimes \dd B + \hat{\Gamma} = \left(
     \int_0^1 B \otimes \circ \dd B - \frac{1}{2} \Sigma \right) + \hat{\Gamma}
  \]
  and compute, using the previously obtained expression for $\hat{\Gamma}$ and
  $\Sigma$,
  \[ \hat{\Gamma} - \frac{1}{2} \Sigma = \tmop{Anti} \left( \sum_{n =
     1}^{\infty} \Delta (n) \right) . \]
     This finishes the proof.
\end{proof}

\begin{example}
  Breuillard-Friz-Huesmann 2009 \cite{breuillard2009random}. The $\xi$ are taken IID, so that $\Gamma=0$, and one
  gets the classical Wong-Zakai / Stratonovich limit.
\end{example}

\begin{example} \label{example:regeneration}
  Orenshtein 2021 \cite{orenshtein2021rough}, Theorem 1.5 (see also the main result of \cite{lopusanschi2021ballistic}). Hereinafter for a process $X$ and a pair of time indices $s<t$ we denote $X_{s,t}:=X_t-X_s$ its increments.   
  %\url{https://arxiv.org/abs/2101.05222}. 
  Here $\xi(n)=X_{n-1,n}$ for $n\in\mathbb N$, where $X_n$ is a delayed regenerative process on $\mathbb R^d$ with respect to the random regeneration times $0=\tau_0<\tau_1<\tau_2<...<\infty$ a.s. That is, the sequence of epochs
\[
\mathcal E_k := \big(
T_{k+1} , (X_{\tau_k,\tau_{k}+\ell})_{\ell=0,1,...,T_{k+1}},
\big)
\]
where $T_0=0$ and $T_{k+1}=\tau_{k+1}-\tau_k$, is independent and identically distributed for $k\ge 1$ and independent of the delay epoch $\mathcal E_0$. 
  % The models studied in the setting of the paper are assumed to be delayed regenerative processes with an optimal moment condition and so the increments are not assumed to be stationary. However, for a stationary version of the model one can easily apply Theorem 1.5 of that paper and deduce a the iterated invariance principles in the rough path topology as long as we assume the first regeneration interval times the corresponding sup norm for the stationary process have a second moment.  

A stationary version $\bar X$ of $X$ is constructed by letting $\bar X_{V,V+n}= \tilde X_{\tau_1, \tau_1 + n}$ for all $n\ge 0$, where $V$ is distributed uniformly on $\{0,1,...,T^*-1\}$, $T^*$ is an independent size-biased version of $T_2=\tau_2-\tau_1$, 
%i.e.~the regeneration interval size, 
and $(\tilde X_{n}, n=0,1,...,T^*-V])$ is an independent copy of $(X_{n}, n=0,1,...,T^*-V])$, c.f. \cite{ney1981refinement, thorisson1983coupling, thorisson1992construction}. In particular, we have that the increments of the stationary process coincide with the ones of a delay-free version of the original process after a time shift of $U=T^*-V$.

In more detail, let $T^*$ be the size-biased version of $T_2$ taken independently of $X$. This is a positive integer-valued random variable whose mass function is
\[
\mathbb P (T^*=m) := \frac{m \mathbb P (T_2=m)}{\mathbb E [T_2]}, m\in \mathbb N.
\]
Let $U$ be uniformly distributed on $\{1,2,...,T^*\}$ and set $V=T^*-U$ (note that it is uniform on $\{0,1,...,T^*-1\}$). 
Let $\tilde X$ be a copy of $X$ independent of anything else and let $\tilde \tau_1,\tilde \tau_2,...$ be the corresponding regeneration times. Finally, a stationary version of $X$ is defined by:
\[
\bar X_n:= 
\begin{cases}
  \tilde X_{\tilde \tau_1+ V+ n}, & \text{if }n\in\{0,1,..,U\},\\ 
  \tilde X_{\tilde \tau_1 + T^*}+ X_{\tau_1,\tau_1+n-U}, & \text{if } n\in\{U,U+1,...\}.   
\end{cases}
\]
By construction, $\bar X=(\bar X_{k})_{k\ge 0}$ is a delayed regenerative process. Let
\[
\Xi(X)(m,n)=\sup \big\{|X_{k,\ell}|: m\le k\le \ell \le n \big\}.
%\vee \sup \big\{|{\mathbb X}^1_{k,\ell}|: m\le k\le \ell \le n \big\}
\]
Assume the process $X$ satisfies the conditions of \cite[Theorem 1.5]{orenshtein2021rough} and assume also the moment condition $\mathbb E[ T_2^2\Xi(X)(\tau_1,\tau_2)^{\otimes 2}] < \infty$. Then the iterated invariance principle holds for $\bar X$ with the same characteristics as $X$. This is an immediate application of \cite[Theorem 1.5]{orenshtein2021rough} to the process $(\bar X_{k})_{k\ge 0}$, after verifying the moment assumption for the delay epoch of $\bar X$.

 To obtain the formulae of \cite{orenshtein2021rough} note that in this case $\bar\xi(n)=\bar X_{n,n+1}$ and
  \[ 
  \bar \Delta (n) : = \tmmathbf{E} [\bar \xi(0) \otimes \bar \xi(n)]. 
  % : = \tmmathbf{E} [X_{\tau_1+ V,\tau_1+ V+1} \otimes X_{\tau_1+ V+n,\tau_1+ V+n+1}]. 
  \]
For $\xi_{\tau_1}(n)=X_{\tau_1+n,\tau_1+n+1}$ we have
\begin{align*}  
   \sum_{n = 1}^{\infty}   \bar \Delta (n) 
   & = \lim_{N\to\infty} \mathbb{E} [\bar \xi (0) \otimes \bar X_{1,N}] \\
   & = \mathbb{E} [\bar \xi (0) \otimes \bar X_{1,U}] \\
   & = \frac{1}{\mathbb E [T_2]}\sum_{m=1}^\infty \mathbb{E} \left[ \sum_{k=0}^{m-1} \sum_{\ell=k+1}^{m-1}  \xi_{\tau_1} (k) \otimes \xi_{\tau_1} (\ell)\right] \mathbb P (T_2=m)   \\
   & = \frac{1}{\mathbb E [T_2]}\mathbb{E} \left[ \sum_{\tau_1\le k<\ell <\tau_2}  \xi (k) \otimes \xi (\ell)\right].
   \end{align*}
Similarly, we find 
  \begin{align*}
    %  \hat{\Gamma} - \frac{1}{2} \Sigma & =
    \tmop{Anti} \left( \sum_{n =1}^{\infty} \bar\Delta (n) \right) 
    & =
    \frac{1}{\mathbb E [T_2]} 
       \mathbb E
   \left[ \tmop{Anti} \left(  \sum_{\tau_1\le k<\ell <\tau_2}  \xi (k) \otimes \xi (\ell)\right)\right]  
  \end{align*}
 and 
  \begin{align*}
% \Sigma  = 
\bar\Delta (0) + 2 \sum_{n = 1}^{\infty} \tmop{Sym} (\bar\Delta (n))
% \\
% & = 
% \tmmathbf{E} \big [X_{\tau_1+V,\tau_1+V+1} \otimes X_{\tau_1+V,\tau_1+V+1})\big ]
%  + 2\tmop{Sym}(\tmmathbf{E} \big [\sum_{n=1}^{U-1} X_{\tau_1+V,\tau_1+V+1} \otimes X_{\tau_1+V+n,\tau_1+V+n+1}\big ])
  % = 
 % \tmmathbf{E} \big [X_{\tau_1+V,\tau_1+T^*} \otimes X_{\tau_1+V,\tau_1+T^*}\big ]
& = 
    \frac{1}{\mathbb E [T_2]} 
       \mathbb E
   \bigg[ \sum_{\tau_1\le k,\ell <\tau_2}  \xi (k) \otimes \xi (\ell)\bigg]  
\\
& = 
   \frac{1}{\mathbb E [T_2]}  \mathbb E
   \big [X_{\tau_1,\tau_2} \otimes X_{\tau_1,\tau_2}\big ].
\end{align*}

\begin{remark}
In \cite{orenshtein2021rough} the moment condition for $X$ is $\mathbb E[ T_{k+1}\Xi(X)(\tau_{k},\tau_{k+1})^p)^{\otimes 2}] < \infty$ for $k=0,1$ and $p=0,1$. However, the moment condition for the delay (that is, for the case $k=0$) can be reduced to $\mathbb E [ \tau_1^{\epsilon}\Xi(X)(0,\tau_{1})^{\epsilon p})^{\otimes 2}]<\infty$ for some $\epsilon>0$ without an essential change in the proof. In particular, for the stationary version, the assumption $\mathbb E[ T_2^{1+\epsilon}\Xi(X)(\tau_1,\tau_2)^{p})^{\otimes 2}] < \infty$ for $p\in\{0,1\}$, which we assumed for $\epsilon=1$ is reduced to any $\epsilon>0$.   
\end{remark}

\end{example}

\subsection{Continuous iterated invariance principles}

\subsubsection{Invariance principles under continuous time assumptions}
Consider a $d$-dimensional continuous-time process $\Xi = \Xi (t) = \Xi (t ;
\omega)$ with rescaled integrals
\begin{equation} \label{eq:bar_S_N}
     \bar{S}_N (t) = \frac{1}{N^{1 / 2}} \int_0^{N t} \Xi (s) \dd s
\end{equation}
canonically enhanced with the
iterated integrals (that is, the integration is taken in the sense of Riemann-Stieltjes), denoted by $\bar{\mathbb{S}}_N$. 

\begin{theorem}
\label{thm:continuous_time}
Assume that $\Xi$ is stationary and that the iterated invariance principle holds for $(\bar S_N,\bar{\mathbb{S}}_N)$ (see \eqref{eq:bar_S_N})
with a Brownian
rough path limit $(\bar{B}, \bar{\mathbb{B}})$
with characteristics $(\bar{\Sigma},
\bar{\Gamma}).$ 
Assume that $\int_0^{\infty} \bar{\Delta} (s) \dd s$ is well-defined and finite, where $\bar{\Delta} (s) : = \tmmathbf{E} (\Xi (0) \otimes \Xi (s))$. Then 
  necessarily we have the characteristics
  \[ \bar{\Sigma} = 2 \tmop{Sym} \left( \int_0^{\infty} \bar{\Delta} (s) \dd s
     \right), \qquad \bar{\Gamma} = \int_0^{\infty} \bar{\Delta} (s) \dd s. \]
  As a consequence, we have the area correction \ \
  \[ \bar{\mathbb{B}} (1) = \int_0^1 B \otimes \circ \dd B + \tmop{Anti} \left(
     \int_0^{\infty} \bar{\Delta} (s) \dd s \right) . \]
\end{theorem}

\begin{proof}
As in the proof of the Theorem \ref{thm:Main_thm_discrete} dealt with the discrete case (e.g., by considering time integrals instead of sums) we use stationarity and integrable correlations and the fact that $\tmmathbf{E} (\bar{S}_N (1)\otimes \bar{S}_N (1)) \approx \tmmathbf{E} ( \bar{B} (1)\otimes \bar{B} (1) )$ to obtain that
\begin{align*}
    \bar{\Sigma} = \tmmathbf{E} (\bar{B} (1)\otimes \bar{B} (1) ) &\approx \frac{1}{N} \int_{0 \leqslant s, s' \leqslant N }
   \tmmathbf{E} (\Xi (s) \otimes \Xi (s')) \dd s \dd s' \\ &\approx \int_0^{\infty}
   \tmmathbf{E} (\Xi (0) \otimes \Xi (s) + \Xi (s) \otimes \Xi (0)) \dd s. 
   \end{align*} 
Hence, in terms of the correlations this becomes
\[ \bar{\Sigma} = 2 \int_0^{\infty} \tmop{Sym} (\bar{\Delta} (s)) \dd s, \]
the Green-Kubo formula. On the other hand, since $\bar{\Gamma} = \tmmathbf{E} (\bar{\mathbb{B}} (1)) \approx \tmmathbf{E} ( \bar{\mathbb{S}}_N (1))$, where

\[ \bar{\mathbb{S}}_N (1) = \frac{1}{N} \int_{0 \leqslant s < s' \leqslant N
   } \Xi (s) \otimes \Xi (s') \dd s \dd s', \]
we have
\[ \bar{\Gamma} \approx \frac{1}{N} \int_{0 \leqslant s < s' \leqslant N }
   \tmmathbf{E} (\Xi (s) \otimes \Xi (s')) \dd s \dd s' \approx \int_0^{\infty}
   \tmmathbf{E} (\Xi (0) \otimes \Xi (s)) \dd s = \int_0^{\infty} \bar{\Delta}
   (s) \dd s. \]
  Next, since
  \[ \bar{\mathbb{B}} (1) \equiv \int_0^1 B \otimes \dd B + \bar{\Gamma} =
     \left( \int_0^1 B \otimes \circ \dd B - \frac{1}{2} \bar{\Sigma} \right) +
     \bar{\Gamma} \]
  and 
  \[ \bar{\Gamma} - \frac{1}{2} \bar{\Sigma} = \int_0^{\infty} (\bar{\Delta}
     (s) - \tmop{Sym} (\bar{\Delta} (s))) \dd s = \int_0^{\infty} \tmop{Anti}
     (\bar{\Delta} (s)) \dd s, \]
  the proof is finished.
\end{proof}

\begin{example}\label{ex:FGL}
  Friz-Gassiat-Lyons 2015 \cite{friz2015physical}, or Ch.3 in Friz-Hairer 14/20 \cite{friz2020course}. Here $\Xi = \Xi (t ;
  \omega)$ is a multidimensional OU process (denoted $\tilde{Y}$ in \cite{friz2020course}) with
  dynamics
  \[ \dd \Xi = - M \Xi \dd t + \dd B ; \]
  a stationary solution exists under a spectral gap condition and is
  explicitly given by
  \[ \Xi (s) = \int_{- \infty}^s e^{- M (s - r)} \dd B (r) . \]
  This allows to compute $\bar{\Delta} (s) : = \tmmathbf{E} (\Xi (0) \otimes
  \Xi (s))$. By It\^{o} isometry,
  \[ \quad \bar{\Delta} (s) = \int_{- \infty}^0 e^{- M (0 - r)} e^{- M^{\star}
     (s - r)} \dd r = : \Sigma_{\tmop{OU}} e^{- M^{\star} s} \]
  and, since $\Sigma_{\tmop{OU}}$ is symmetric,
  \[ (\star) : \int_0^{\infty} \bar{\Delta} (s) \dd s = \Sigma_{\tmop{OU}}
     (M^{\star})^{- 1} \]

  Recall that $\bar{S}_N (t) = \frac{1}{N^{1 / 2}} \int_0^{N t} \Xi (s) \dd s$
  represents physical Brownian motion with mass $1 / N$.
  
  By Theorem~\ref{thm:continuous_time}, the limit Brownian covariance equals
  \[ 2 \tmop{Sym} \left( \int_0^{\infty} \bar{\Delta} (s) \dd s \right) = \Sigma_{\tmop{OU}}
     (M^{\star})^{- 1} + (M )^{- 1} \Sigma_{\tmop{OU}} 
     %= \cdots = \tmcolor{red}{M^{- 1} ? ? ?
     %{\color[HTML]{000000}}} 
     \]
  while we have area correction given by
  \[ \tmop{Anti} \left( \int_0^{\infty} \bar{\Delta} (s) \dd s \right) =
     \frac{1}{2} [\Sigma_{\tmop{OU}} (M^{\star})^{- 1} - (M )^{- 1} \Sigma_{\tmop{OU}}] \]
  which vanishes iff $M$=$M^{\star}$. 

% \mnote{I do not see what the problematic point is here. Why do we write $N$ instead of $\Sigma_{\tmop{OU}}$ and what is that supposed to help or even mean given that we consider $N \to \infty$? For the last remark: why is it obvious that $\Sigma_{\tmop{OU}}$ and $M$ commute if $M$ is symmetric?}

% \tnote{Replying to Max: Why do we write $N$ instead of $\Sigma_{\tmop{OU}}$ and what is that supposed to help or even mean given that we consider $N \to \infty$? 
% I think is is a typo. 
% For the last remark: why is it obvious that $\Sigma_{\tmop{OU}}$ and $M$ commute if $M$ is symmetric? Answer: $\Sigma_{\tmop{OU}}$ commutes with $M+M^*$ since it is an integral of exponents of $M+M^*$ (consider an exponent of a matrix $sM+sM^*$ as an infinite weighted sum of powers of $sM+sM^*$)}

  % \tnote{(comment written in 2022:) Here is a retry: multiplying first by $M$:\\ 
  % Write $\bar{S}_N (t) = \frac{1}{N^{1 / 2}} \int_0^{N t} M \Xi (s) \dd s$.
  % Note that $\bar{\Delta} (0) =  \Sigma_{\tmop{OU}}$.  Together with Theorem 3.8 in FH we get that the limit Brownian covariance equals \[ 2 \tmop{Sym} \left( \int_0^{\infty}  M \bar{\Delta} (s) \dd s \right) =\tmop{Sym}(M\bar \Delta(0)) \]while we have area correction given by \[ \tmop{Anti} \left( \int_0^{\infty} \bar{\Delta} (s) \dd s \right) =
  %    \tmop{Anti} \left( M\bar \Delta(0)\right)\]which vanishes iff $M$=$M^{\star}$ since $\bar \Delta(0)$ is symmetric.      
  %    }

\end{example}

\begin{example}
%  Deuschel et al AOP 2021 \url{https://arxiv.org/pdf/1912.09819.pdf} 
  Deuschel-Orenshtein-Perkowski %AOP 2021
  \cite{deuschel2021additive} consider
  \[ \Xi (s) = F (X (s)), \]
  for some Markov process $X$ on some Polish space with generator $\mathcal{L}$ on some space $M$ and a probability measure $\pi$ which is invariant and ergodic for both $\mathcal{L}$ and $\mathcal{L}^*$, see the precise formulation in \cite[Chapter 3]{deuschel2021additive}. Assume that
  $\tmmathbf{E} F (X (0)) = 0$ and that exists a solution in $L^2$ to the Poisson equation $-\mathcal{L} \phi = F$, they first recall
  (cf. p.7) validity of functional CLT, with ``half-covariance'' given by
  \[ \frac{1}{2} \bar{\Sigma} = \tmmathbf{E}  (\phi (X (0)) \otimes
     ((-\mathcal{L}_S) \phi) (X (0))) \]
  and then establish (p.6, Theorem 3.3) an area correction of the form
  \[ \tmmathbf{E}  (\phi (X (0)) \otimes (\mathcal{L}_A \phi) (X (0))
     \nobracket . \]
  Let us illustrate how this can be derived from Theorem~\ref{thm:continuous_time}. We start by
  computing
  \[ \bar{\Gamma} \hspace{1em} \equiv \int_0^{\infty} \bar{\Delta} (s) \dd s
     \equiv \int_0^{\infty} \tmmathbf{E} (\Xi (0) \otimes \Xi (s)) \dd s \]
  using the Markov structure at hand. With $\Xi (s) = F (X (s))$ we have
  \begin{eqnarray*}
    \bar{\Gamma} = \int_0^{\infty} \left( \int_M F \otimes e^{s\mathcal{L}} F
    \dd \pi \right) \dd s & = & \int_M \left( F \otimes \left( \int_0^{\infty}
    e^{s\mathcal{L}} F \dd s \right) \dd \pi \right)\\
    & = & \int_M (F \otimes (-\mathcal{L}^{- 1} F) \dd \pi)\\
    & = & \int_M (-\mathcal{L} \phi \otimes \nobracket \phi ) \dd \pi
    \nobracket .
  \end{eqnarray*}
  The last two steps are justified when $\mathcal{L}$ has a spectral gap and
  $-\mathcal{L} \phi = F$ is uniquely solvable. Decomposing $\mathcal{L}=\mathcal{L}_S +\mathcal{L}_A =
  \frac{1}{2} (\mathcal{L}  +\mathcal{L}^{\star}) + \frac{1}{2} (\mathcal{L} 
  -\mathcal{L}^{\star})$ into symmetric and antisymmetric operators, we see
  \[ \tmop{Sym} (\bar{\Gamma}) = \tmop{Sym} \left( \int_M (-\mathcal{L} \phi
     \otimes \nobracket \phi ) \dd \pi \nobracket \right) = \int_M
     (-\mathcal{L}_S \phi \otimes \nobracket \phi ) \dd \pi = \nobracket \int_M
     (\phi \otimes -\mathcal{L}_S \nobracket \phi ) \dd \pi \nobracket \]
  and
  \[ \tmop{Anti} (\bar{\Gamma}) = \tmop{Anti} \left( \int_M (-\mathcal{L} \phi
     \otimes \nobracket \phi ) \dd \pi \nobracket \right) = \int_M
     (-\mathcal{L}_A \phi \otimes \nobracket \phi ) \dd \pi = \nobracket \int_M
     (\phi \otimes \mathcal{L}_A \nobracket \phi ) \dd \pi \nobracket . \]
  These expressions agree exactly with the half-covariance and area correction
  given by Deuschel-Orenshtein-Perkowski \cite{deuschel2021additive}. The more general case, Theorem 3.3 of \cite{deuschel2021additive}, is a Kipnis-Varadhan type invariance principle, where the assumption of existence of solution to the Poisson equation is weaken to a condition on the behavior of the solution $\pi_\lambda$ to the resolvent equation $(\lambda-\mathcal{L}) \phi_\lambda = F$ for small $\lambda$, known as the $H^{-1}$ condition (for the precise formulation see line (3.1) of \cite{deuschel2021additive}). In those settings $\bar{\Gamma}$ is obtained by the limit as $\lambda \downarrow 0$ of the Laplace transform of the correlations. More precisely,     
    \begin{eqnarray*}
    \bar{\Gamma} & = & 
%    \int_0^{\infty} \tmmathbf{E} (\Xi (0) \otimes \Xi (s)) \dd s
%    \\
%    & = & 
    \lim_{\lambda \downarrow 0 }\int_0^{\infty} e^{-\lambda s}\tmmathbf{E} (\Xi (0) \otimes \Xi (s)) \dd s \\
    & = & 
    \lim_{\lambda \downarrow 0 }\int_0^{\infty} e^{-\lambda s}\int \tmmathbf{E} (\Xi (0) \otimes \Xi (s)|X(0)=x) \dd \pi(x) \dd s 
    \\
    & = & 
    \lim_{\lambda \downarrow 0 }
    \int F(x)  \otimes \left( \int_0^{\infty} e^{-\lambda s}  \tmmathbf{E}(F(X(s))|X(0)=x) \dd s  \right)  \dd \pi(x) \\
    & = & 
    \lim_{\lambda \downarrow 0 }
    \int F \otimes \phi_\lambda  \dd \pi \\
    & = & 
    \lim_{\lambda \downarrow 0 }
    \int (-\mathcal{L} \phi_\lambda \otimes \nobracket \phi_\lambda ) \dd \pi, 
     \end{eqnarray*}
  where the existence of the limits are justified as in \cite{deuschel2021additive}, in particular the last equality follows as the one in pp.\ 1463 second line in that paper.  
\end{example}

\begin{example} (Example
 \ref{ex:FGL} revisited)
 We can view Example
 \ref{ex:FGL} from a Markov / Deuschel-Orenshtein-Perkowski \cite{deuschel2021additive}
  perspective by considering the generator 
  $$\mathcal{L}f(x)= - M x\cdot \nabla f(x) + \frac{1}{2} \Delta f(x)$$
  for $f \in C^2(\mathbb R^d, \mathbb R^d)$, where the operator acts in the usual way in each component. The
  Poisson equation
  \[ -\mathcal{L} \phi  = x \in \mathbb{R}^d \]
  is solved by $\phi (x)  = M^{- 1} x$
  centred w.r.t.\ the invariant measure $\pi = N (0, \Sigma_{\tmop{OU}}
  \nobracket$). Note that $\mathcal{L}$ is self-adjoint if and only if $M$ is symmetric and
  \[ \begin{array}{ll}
       & \int  (-\mathcal{L} \phi \otimes \nobracket \phi ) \dd \pi \nobracket
     \end{array} = \int  (x \otimes \nobracket M^{- 1} x ) \dd \pi \nobracket =
     \int  (x \nobracket x^{\star} ) \dd \pi \nobracket (M^{\star})^{- 1} =
     \Sigma_{\tmop{OU}} (M^{\star})^{- 1} \]
  which is precisely $(\star)$ above and in particular
    \[ \tmop{Anti} (\bar{\Gamma}) = \tmop{Anti} \left( \int_M (-\mathcal{L} \phi
     \otimes \nobracket \phi ) \dd \pi \nobracket \right) 
     = 
     %\int_M (-\mathcal{L}_A \phi \otimes \nobracket \phi ) \dd \pi = 
     \nobracket \int_M
     (\phi \otimes \mathcal{L}_A \nobracket \phi ) \dd \pi \nobracket  \]
vanishes if $\mathcal{L}$ is self-adjoint, which holds if and only if $M=M^*$. 
\end{example}
  
\subsubsection{Invariance principles for suspension flows under discrete time assumptions}

We consider the following setting (cf.~\cite{friz2021almost}) as specification of the previous section, allowing for weaker assumptions on the integrability of  correlations $\Delta(t)$. 

%{\color{blue} (ME: I leave out the question of two-sided or one-sided time here; could be added if necessary.)}

Take a complete probability space $(\Omega, \mathcal F, \mathbb P)$ with $\mathbb P$-preserving transformation $\theta: \Omega \to \Omega$ and a \emph{roof function} $\tau: \Omega \to (0, \infty)$ which is bounded from below and above, i.e.~
$$ L^{-1} \leq \tau \leq L, \quad \text{ for some } L > 0.$$
We consider the probability space $(\bar \Omega, \bar{\mathcal F}, \bar{\mathbb P})$ with
$$ \bar \Omega = \{ (\omega, t) \, : \, \omega \in \Omega, \, 0 \leq t \leq \tau(\omega), \, (\omega, \tau(\omega)) = (\theta \omega, 0) \}, $$
where $\bar{\mathcal F}$ is the restriction to $\bar \Omega$ of $\mathcal F \times \mathcal B([0, L])$ and $B([0, L])$ denotes the Borel $\sigma$-algebra. Denoting $\bar \tau = \int \tau \dd \mathbb P$, the measure $\bar{\mathbb P}$ is given, for any $ A \in \bar{\mathcal F}$, by
$$ \bar{\mathbb P}(A) =  \frac{1}{\bar \tau} \int_{\bar \Omega} \mathds 1_A (\omega, t) \dd \mathbb P(\omega) \dd t.  $$
Now we can introduce the $d$-dimensional continuous-time process $\Xi$ on $\bar \Omega$ as
\begin{align*}
    \Xi (t; (\omega, s)) = \Xi(t+s; (\omega, 0)) = \Xi(0; (\omega, t+s)) \ \text { if } 0 \leq t+s < \tau(\omega), \\
    \Xi(t; (\omega, s)) = \Xi(0; (\theta^k \omega, u)), \text { if } t+s = u + \sum_{j=0}^k \tau(\theta^j \omega), \, 0 \leq u < \tau (\theta^k \omega).
\end{align*}
Note that the \emph{suspension} $\Xi$ with respect to the map $\theta$ is a stationary process on $(\bar \Omega, \bar{\mathcal F}, \bar{\mathbb P})$, writing $\Xi(t; \omega)$ for $\Xi(t; (\omega, 0))$.

We can now introduce $\xi(0; \omega) = \int_0^{\tau(\omega)} \Xi(s; \omega) \dd s$ and $\xi(k; \omega) = \xi(0; \theta^k(\omega))$ for $k \geq 0$ and take the partial sums $(S_N, \mathbb{S}_N)$ for $\xi$ as in~\eqref{eq:S_N} and the partial sums $(\bar{S}_N, \bar{\mathbb{S}}_N)$ for $\Xi$ as in~\eqref{eq:bar_S_N}. We consider the following situation (which is, for example, satisfied in \cite[Theorem 6.1]{kelly2016smooth} or \cite[Theorem 2.5]{friz2021almost}), again denoting $\Delta (n)  = \tmmathbf{E} (\xi (0) \otimes \xi (n))$:
\begin{assumption}
 \begin{itemize}
     \item[(a)] The iterated invariance principle holds for $(S_N, \mathbb{S}_N)$ with Brownian rough path limit $(B, \mathbb B)$ and characteristics
 \[ \Sigma = \Delta (0) + 2 \tmop{Sym} (  \Gamma), \qquad \Gamma = \sum_{n = 1}^{\infty} \Delta (n)  \]   
     according to Theorem~\ref{thm:Main_thm_discrete}, i.e.~the summability assumption is satisfied.

     \item[(b)] The iterated invariance principle holds for $(\bar{S}_N, \bar{\mathbb{S}}_N)$ with Brownian rough path limit $(\bar B, \bar{\mathbb B})$ such that
     \begin{align*}
         \bar B &= (\bar \tau)^{-1/2} B, \\
         \bar{\mathbb B} (t) &= \int_0^t B  \otimes \dd B + t (\bar \tau)^{-1} \Gamma + t \int_{\bar \Omega} \left(\int_0^{u} \Xi(s; \omega) \dd s \right) \otimes \Xi (0; \omega) \, \dd \bar{\mathbb P}.
     \end{align*}
 \end{itemize}
\end{assumption}

We can now state our final theorem:
\begin{theorem}
\label{thm:suspension}
  Under Assumption 1, the Brownian rough path limit $(\bar B, \bar{\mathbb{B}})$ has the characteristics
  \[ \bar{\Sigma} = (\bar \tau)^{-1} \Sigma, \qquad \bar{\Gamma} =  (\bar \tau)^{-1} \Gamma + \int_{\bar \Omega} \left(\int_0^{u} \Xi(s; \omega) \dd s \right) \otimes \Xi (0; \omega) \, \dd \bar{\mathbb P}, \]
  and  the area correction 
  \[ \bar{\mathbb{B}} (1) = \int_0^1 B \otimes \circ \dd B +  \tmop{Anti} \left( \bar \Gamma \right)
     %\tmop{Anti} \left( \int_{\bar \Omega} \left(\int_0^{u} \Xi(s, \omega) \dd s \right) \otimes \Xi (0, \omega) \, \dd \bar{\mathbb P}(\omega, u) \right)
     .
     \]
\end{theorem}
\begin{proof}
The expression for the covariance matrix follows directly from the assumption.
Concerning the area correction: as before, we need to compute $\bar \Gamma - \frac{1}{2} \bar \Sigma$. We have
\begin{align*}
    \bar \Gamma - \frac{1}{2} \bar \Sigma &= (\bar \tau)^{-1} \left(\sum_{n=1}^{\infty} \Delta(n) - \tmop{Sym}(\Delta(n)) - \frac{1}{2} \Delta(0) \right) \\
    &+ \int_{\bar \Omega} \left(\int_0^{u} \Xi(s; \omega) \dd s \right) \otimes \Xi (0; \omega) \, \dd \bar{\mathbb P}(\omega, u) \\
    &= (\bar \tau)^{-1}  \tmop{Anti} \left( \Gamma \right) \\
    &+ \left(\int_{\bar \Omega} \left(\int_0^{u} \Xi(s; \omega) \dd s \right) \otimes \Xi (0; \omega) \, \dd \bar{\mathbb P}(\omega, u) - \frac{1}{2} (\bar \tau)^{-1}  \Delta(0)  \right).
\end{align*} 
It remains to show that 
$$ (\bar \tau)^{-1}  \Delta(0)  = 2 \tmop{Sym} \left( \int_{\bar \Omega} \left(\int_0^{u} \Xi(s; \omega) \dd s \right) \otimes \Xi (0; \omega) \, \dd \bar{\mathbb P}(\omega, u)  \right).  $$
We observe this as follows:
\begin{align*}
    \Delta (0) &= \tmmathbf{E} (\xi (0) \otimes \xi (0) ) = \int_{\Omega} \left(\int_0^{\tau(\omega)} \Xi(u; \omega) \dd u \right) \otimes \left(\int_0^{\tau(\omega)} \Xi(s; \omega) \dd s \right) \dd \mathbb P \\
    &= \int_{\Omega} \left(\int_0^{\tau(\omega)} \Xi(u; \omega) \dd u \right) \otimes \left( \int_0^{u} \Xi(s; \omega) \dd s + \int_u^{\tau(\omega)} \Xi(s; \omega) \dd s \right) \dd \mathbb P \\
    &= \int_{\Omega} \int_0^{\tau(\omega)} \Xi(u; \omega) \otimes \left( \int_0^{u} \Xi(s; \omega) \dd s \right) \dd u \, \dd \mathbb P \\
    &+ \int_{\Omega} \int_0^{\tau(\omega)} \left( \int_0^{s} \Xi(u; \omega) \dd u \right) \otimes \Xi(s; \omega) \, \dd s \, \dd \mathbb P \\
    &= \int_{\Omega}  \Xi(0; \omega) \otimes \left( \int_0^{u} \Xi(s; \omega) \dd s \right) \dd \bar{\mathbb P} + \int_{\Omega} \left( \int_0^{s} \Xi(u; \omega) \dd u \right) \otimes \Xi(0; \omega) \, \dd \bar{\mathbb P}.
\end{align*}
This concludes the proof.
\end{proof}

\begin{example}
Again, we refer to Kelly-Melbourne  \cite{kelly2016smooth, kelly2017deterministic} and Chevyrev et al.~\cite{chevyrev2019multiscale}.
  %\url{https://arxiv.org/abs/1403.7281}.
  Theorem~\ref{thm:continuous_time} applies with
  \[ \bar{\Delta} (s) : = \tmmathbf{E} (\Xi (0) \otimes \Xi (s)) = \int (v
     \otimes (v \circ \phi_s)) \dd \nu , \]
  for a flow $\phi$ on some compact manifold $M$ with invariant and ergodic measure $\nu$ on a hyperbolic invariant set $\Omega$, and for a $\nu$-centered H\"older observable $v : \Omega \rightarrow
  \mathbb{R}^d$. The probability space is thus $(\Omega, \nu)$ and $\Xi (s ;
  \omega) = (v \circ \phi_s) (\omega)$. In particular, the covariance and area
  corrections from Theorem~\ref{thm:continuous_time} agree with those given in \cite[Theorem 1.1 (b) and Corollary 8.3]{kelly2016smooth},
  whenever these integrals are finite. 
  The key examples are semi-flows which are suspensions of the uniformly and non-uniformly expanding maps mentioned in Example 1, and flows that are suspensions of uniformly and
non-uniformly hyperbolic diffeomorphisms. This includes Axiom A flows, which are typically mixing for $v \in C^{\infty}$ \cite{fieldmelbournetoeroek07} and by that satisfy the assumptions for Theorem~\ref{thm:continuous_time}, and non-uniformly hyperbolic flows constructed as suspensions
over Young towers with exponential tails. 

However, note that the formulas in Theorem~\ref{thm:continuous_time} are only valid if the (semi-)flows are (sufficiently fast) mixing themselves; otherwise the covariance and area correction can be given in terms of iterations of the associated Poincar\'{e} map $T:\Lambda \to \Lambda$, with ergodic, invariant measure $\mu$ on $\Lambda$, precisely in terms of Theorem~\ref{thm:suspension}, assuming that $T$ is mixing such that Theorem~\ref{thm:Main_thm_discrete} holds  (see also \cite[Corollary 8.1]{kelly2016smooth}); in more detail, 
$$ \bar \Omega = \{ (\omega, t) \, : \, \omega \in \Lambda, \, 0 \leq t \leq \tau(\omega), \, (\omega, \tau(\omega)) = (T \omega, 0) \}, $$
such that $\nu = (\bar \tau)^{-1} (\mu \times \text{Lebesgue})$ is an ergodic invariant probability measure for $\phi_s(\omega,u) = (\omega, u+s)$, where $\bar \tau = \int_{\Lambda} \tau \dd \mu$. 
Then the summability condition applies to
 \[ \Delta (n) : = \tmmathbf{E} (\xi (0) \otimes \xi (n)) = \int (\bar v \otimes
     (\bar v \circ T^n)) \dd \mu, \]
  where 
  $$\bar v(\omega) = \int_0^{\tau(\omega)} v(\omega, s) \dd s = \int_0^{\tau(\omega)} (v \circ \phi_0) (\omega, s)  \dd s =  \int_0^{\tau(\omega)} \Xi (s; \omega)  \dd s $$ 
  and the stationary processes are $\xi (n ; \omega) = (\bar v \circ T^n) (\omega)$ and $\Xi(s;\omega) = (v \circ \phi_s) (\omega, 0) =  (v \circ \phi_0) (\omega, s) $.
\end{example}

\begin{example}
    Friz-Kifer 2021 \cite{friz2021almost}.
    % Theorem 2.5 \url{https://arxiv.org/abs/2111.05390}. 
    Rescaling time via $ t \to \bar \tau t$, the process $V^{\varepsilon}$ therein fulfills Assumption 1 according to  \cite[Theorem 2.5]{friz2021almost}; note that our notation $\xi$ corresponds with their $\eta$ and our notation $\Xi$ with their $\xi$.
   Theorem~\ref{thm:suspension} can be applied to obtain the area correction. 
   % \mnote{(I could not directly find the area correction formula in the paper so it seems we are adding something.)}
    Note that in \cite{friz2021almost} the term 
    $$\int_{\bar \Omega} \left(\int_0^{u} \Xi(s; \omega) \dd s \right) \otimes \Xi (0; \omega) \, \dd \bar{\mathbb P}$$
    as a summand in $\bar{\mathbb B}(1)$, and thereby also in $\bar \Gamma$,
    is expressed as
    $$ \tmmathbf{E} \left( \int_0^{\tau(\omega)} \Xi(s; \omega) \dd s \otimes \int_0^{s} \Xi (u; \omega) \, \dd u \right),$$
    i.e.~in terms of the invariant measure $\mathbb P$ for the discrete time system.
    
 % ... and has limiting covariance given in (2.22) as
 %  \[ \frac{1}{2} \Sigma = \frac{1}{2} \tmmathbf{E} (\eta (0) \otimes \eta (0))
 %     + \tmop{Sym} \left( \sum_{n = 1}^{\infty} \tmmathbf{E} (\eta (0) \otimes
 %     \eta (n)) \right) \]
 %  with discrete auxilary process $\eta$. Also an area correction is given.
 %  Should be possible to derive both from {\textdots}
 %  \[ \tmmathbf{E} \int_0^1 V^{\varepsilon} \otimes \dd V^{\varepsilon} =
 %     \varepsilon^2 \int_{0 \leqslant s, s' \leqslant \varepsilon^{- 2}} \xi
 %     (s) \otimes \xi (s') \dd s \dd s' \]
\end{example}

{\centering 
\subsubsection*{Acknowledgements}}
ME acknowledges the support of Deutsche Forschungsgemeinschaft (DFG) through CRC 1114 and under Germany's Excellence Strategy -- The Berlin Mathematics Research Center MATH+ (EXC-2046/1, project 390685689); and additionally thanks the DFG-funded SPP 2298 and the Einstein Foundation for supporting his research.
PFK acknowledges the support of Deutsche Forschungsgemeinschaft (DFG) through a MATH+ Distinguished Fellowship. The research of TO was partially sponsored by INDAM—GNAMPA Project codice CUP-E53C22001930001. 

%}

\bibliographystyle{plain} % We choose the "plain" reference style
\bibliography{IIP-bib} % Entries are in the refs.bib file

\end{document}